\documentclass[12pt,reqno]{amsart}
\usepackage{enumerate, latexsym, amsmath, amsfonts, amssymb, amsthm, graphicx, color}
\def\pmod #1{\ ({\rm{mod}}\ #1)}
\def\Z{\Bbb Z}

\def\Q{\Bbb Q}
\def\zp{\Z^+}

\def\bg{\bigg}
\def\({\bg(}
\def\){\bg)}
\def\t{\text}

\def\eq{\equiv}
\def\Da{\Delta}

\theoremstyle{plain}
\newtheorem{theorem}{Theorem}

\newtheorem{lemma}{Lemma}
\newtheorem{corollary}{Corollary}

\theoremstyle{definition}

\theoremstyle{remark}
\newtheorem{remark}{Remark}

\makeatletter
\@namedef{subjclassname@2020}{%
  \textup{2020} Mathematics Subject Classification}
\makeatother
 \vspace{4mm}

\begin{document}

\title[Legendre Symbol of $\prod f(i,j) $ over $ 0<i<j<p/2,  \ p\nmid f(i,j) $ ]{Legendre Symbol of $\prod f(i,j) $ over $ 0<i<j<p/2,  \ p\nmid f(i,j) $}

\author[Chao Huang]{Chao Huang}
\address{(Chao Huang) Department of Mathematics\\ Nanjing University\\
Nanjing 210093, People's Republic of China}
\email{dg1921004@smail.nju.edu.cn}

\begin{abstract}
Let $p>3$ be a prime. We investigate Legendre Symbol of $\displaystyle \prod_{0<i<j<p/2 \atop p\nmid f(i,j) } f(i,j) \ $, where $i,j\in \Bbb Z$, f(i,j) is a linear or quadratic form with integer coefficients. When $f=ai^2+bij+cj^2$ and $p\nmid c(a+b+c)$ , we prove that to evaluate the product is equivalent to determine $ \displaystyle \sum_{y=1}^{p-1} \bigg(\frac{y(y+1)(y+k)}{p}\bigg) \pmod{16}$ , where $4c(a+b+c)k \eq (4ac-b^2)\pmod{p}$. Parallel results are given for $\displaystyle \prod_{i,j=1 \atop p\nmid f(i,j) }^{(p-1)/2} \bigg(\frac{ f(i,j) }{p}\bigg).$  Then we show that $ \displaystyle \sum_{y=1}^{p-1} \bigg(\frac{y(y+1)(y+k)}{p}\bigg) \pmod{16}$ can be evaluated explicitly when k=2,4,5,9,10 or k is a square. And for several classes of f(i,j) these two kinds of products can be evaluated explicitly. Finally when f is a linear form we give unified identities for these products. Thus we prove these kind of problems  problems raised in \cite{S20}.

\end{abstract}



\maketitle

\section{Introduction}
Let p be an odd prime and $a,b,c \in \Z .$  Zhi-Wei Sun has given evaluations and identities of $\displaystyle \prod_{1\leq i<j \leq p-1 \atop p\nmid ai^2+bij+cj^2 } ai^2+bij+cj^2 \ $ , \ $\displaystyle \prod_{i,j=1 \atop p\nmid ai^2+bij+cj^2 }^{(p-1)/2} ai^2+bij+cj^2 \ $, $\displaystyle \prod_{1\leq i<j \leq p-1 \atop p\nmid i^2+j^2 } \sin \pi \frac{a(j^2+k^2)}{p} \ $ and many others in studying permutations related to quadratic residues.\cite{S19} \cite{S20} \\

Sun also raised the problem concerning product of Legendre symbols $\displaystyle \prod_{0<i<j<p/2 \atop p\nmid f(i,j) } \bigg(\frac{f(i,j)}{p}\bigg) \ $, where f(i,j) is linear or quadratic integral forms.\cite{S20} He gave evaluations and identities of this kind of products for many specific f(i,j) as conjectures. In this paper we study this kind of products of Legendre symbols. Note that this product takes only $\pm1$ since we always ignore those $(i,j)$ with $p\mid f(i,j)$. \\

Our Theorem 2.3 show that when $f =ai^2+bij+cj^2$ and $p\nmid c(a+b+c)$, to evaluate $\displaystyle \prod_{0<i<j<p/2 \atop p\nmid f(i,j) } \bigg(\frac{f(i,j)}{p}\bigg) \ $ is equivalent to determine the residue $ \displaystyle \sum_{y=1}^{p-1} \bigg(\frac{y(y+1)(y+k)}{p}\bigg) \  \pmod{16}$, where $k\in \Z$ with $4c(a+b+c)k \eq (4ac-b^2)\pmod{p}.$ In Theorem 2.4  there are parallel results for $\displaystyle \prod_{i,j=1 \atop p\nmid f(i,j) }^{(p-1)/2} \bigg(\frac{ f(i,j) }{p}\bigg) \ ,$ with $4ack \eq b^2 \pmod{p}.$ \\

 As a result, we can classify these two kinds of products according to k and give many general results. In Section 3 we show that when $ k=2,4,5,9,10$ or k is a square number, viewed as a function of p for fixed irreducible f(i,j), both $ \displaystyle \sum_{y=1}^{p-1} \bigg(\frac{y(y+1)(y+k)}{p}\bigg) \  \pmod{16}$ and all the products equivalent to it, are periodic and can be evaluated explicitly. Moreover, when f(i,j) is reducible in $\Z$, the products will always be periodic of p and can be evaluated explicitly. Thus we solve Sun's Conjectures 7.1, 7.2, 7.3 and 7.5.  \cite{S20} \\

Next we consider products for linear forms. Let $ \# N_p(s) $ denotes the number of non-residues in intervals $\displaystyle \mathop{\cup}_{k=1}^{ \lfloor s/2  \rfloor } \big(\frac{(2k-1)p}{2s}, \frac{(2k)p}{2s}\big).$ In Theorem 4.2 we give identities involving $ \# N_p(s) $  for $ \displaystyle \prod_{i,j=1  \atop p \nmid si +\epsilon j }^{ (p-1)/2}  \bigg(\frac{si +\epsilon j}{p} \bigg)$, where $\epsilon=\pm1$. Parallel results can be obtained for $\displaystyle \prod_{0<i<j<p/2 \atop p\nmid si+\epsilon j } \bigg(\frac{si +\epsilon j}{p}\bigg).  $\\

 Finally we show our identities for $ s= 3,4,5,6,8$ are equivalent to those in Conjectures 7,6-7.10 \cite{S20}, using results concerning symmetries for sums of the Legendre symbol.

\section{Main Theorems}

\setcounter{lemma}{0}
\setcounter{theorem}{0}
\setcounter{corollary}{0}
\setcounter{remark}{0}
\setcounter{equation}{0}
\setcounter{conjecture}{0}

\begin{theorem}
\ Let $p > 3$ be a prime. We use symbol $ \{ \}_{p}$ to denote least residue modulo p and the cardinal of any set A will be denoted $| A |$ . For $ j \in \{2,3,\ldots p-1 \}$ , let
$$M_p(j):=\{i \in \{\ 1,2, \ldots
 \frac{p-1}{2}\}| \ i < \{\ ij\}_{p} < \frac{p}{2} \}.$$
 Then we have
 $$ |M_p(j)| \eq \frac{\big(\frac{1-j}{p}\big) -1 }{2}+\frac{p^2-1}{8} \pmod{2} ,$$
 which amounts to
 $$\begin{cases} \big(\frac{j-1}{p}\big) = (-1)^{|M_p(j)|}  &\t{when}\ p\eq1,3\pmod{8},
\\ \big(\frac{j-1}{p}\big) = (-1)^{|M_p(j)|+1}   &\t{when}\ p\eq5,7 \pmod{8}.
\end{cases}$$.
\end{theorem}

\proof
As $ i < \frac{p}{2}$ , if  $ i < \{\ ij\}_{p} < \frac{p}{2}$ , then $ \{\ i(j-1)\}_{p} < \frac{p}{2}.$  So clearly we have
$$ |M_p(j)|=|\{i \in \{\ 1,2, \ldots
 \frac{p-1}{2}\}| \  \{\ i(j-1)\}_{p} < \frac{p}{2} \}|   $$
$$ \\ - |\{i \in \{\ 1,2, \ldots
 \frac{p-1}{2}\}|  \{\ i(j-1)\}_{p}  < \frac{p}{2} \ <  \{\ i \}_{p} \}| $$
 On the right the first term's parity is the same as $\frac{\big(\frac{j-1}{p}\big) -1 }{2}+ \frac{p-1}{2}$ by Gauss' lemma, while the second is treated in the next lemma.  \qed

\setcounter{lemma}{1}
\begin{lemma}
\ Let $p > 3$ be a prime. For $ j \in \{2,3,\ldots p-1 \}$ , let
$$L_p(j):=\{i \in \{\ 1,2, \ldots
 \frac{p-1}{2}\} \ | \  \{\ i(j-1) \}_{p} < \frac{p}{2} <  \{\ ij\}_{p} \}.$$
 Then we have
 $$ |L_p(j)| \eq \frac{p^2-1}{8} \pmod{2} .$$
\end{lemma}

\proof
For fixed $j$, and $i \in \{\ 1,2, \ldots
 \frac{p-1}{2}\},$ let $a_i = i(j-1)$ and  $ b_i = ij.$ As $ i < \frac{p}{2},$ we have $ \{\ a_i\}_{p} <  \{\ b_i\}_{p} .$ And there are four cases for i: \\  \textcircled{1}  $ \{\ a_i\}_{p} <  \{\ b_i\}_{p} < \frac{p}{2} ,$
  \textcircled{2}  $ \{\ a_i\}_{p} < \frac{p}{2} < \{\ b_i\}_{p} ,$
   \textcircled{3}  $  \frac{p}{2} <  \{\ a_i\}_{p} <  \{\ b_i\}_{p} ,$ \\
    \textcircled{4}  $ \{\ b_i\}_{p} < \frac{p}{2} <  \{\ a_i\}_{p} ,$ in which \textcircled{2} is $L_p(j)$ stands for. \\

     Now we turn to consider smallest residue in absolute value modulo p. For any $x \in \Z$, we use symbol $<x>$ to denote the unique integer with $x \eq <x> \pmod{p}$ and $<x> \in (-p/2,p/2)$. It's easily checked that in other three cases $ a_i - < a_i > =  b_i - < b_i>,$ whereas in \textcircled{2}, $ a_i - < a_i > =  b_i - < b_i> - p.$ So
     $$\sum_{i=1}^\frac{p-1}{2} \big( a_i-<a_i> \big)=  \sum_{i=1}^\frac{p-1}{2} \big( b_i-<b_i> \big)  - p|L_p(j)|. $$

    For each i, either $<a_i>$ or $ -<a_i> $ is in $\{\ 1,2, \ldots
 \frac{p-1}{2}\}$. And for $i_1 \neq i_2 $, we have $ < a_{i_1}> \not \eq  \pm < a_{i_2}>\pmod{p}.$ Thus
    $$\sum_{i=1}^\frac{p-1}{2} <a_i> \ \eq  \sum_{i=1}^\frac{p-1}{2} i \  \eq  \sum_{i=1}^\frac{p-1}{2} <b_i> \  \pmod{2}. $$
And by definition of $a_i,b_i,$ we have
   $$\sum_{i=1}^\frac{p-1}{2} b_i\ -  \sum_{i=1}^\frac{p-1}{2} a_i \  =  \sum_{i=1}^\frac{p-1}{2} i \ = \frac{p^2-1}{8} \pmod{2} .$$
    Therefore $ p|L_p(j)| \eq \frac{p^2-1}{8} \pmod{2} ,$ and the lemma follows as p is odd.               \qed \\

\begin{remark}
Theorem 2.1 deals with the upper-right triangle of $ (0 , p/2) \times (0 , p/2) $ .We can also consider it in upper-left triangle of $ (0 , p/2) \times (0 , p/2) $
$$| \{i \in \{\ 1,2, \ldots
 \frac{p-1}{2}\}| \  \{\ ij\}_{p} < \ \{\ ij\}_{p} +i  < \frac{p}{2} \} |$$
 and in upper-right of  $ (0 , p/2) \times (0 , p) $
 $$| \{i \in \{\ 1,2, \ldots
 \frac{p-1}{2}\}| \ 2i < \  \{\ ij\}_{p} \} |.$$
For fixed j, their parities depend on $\big(\frac{j}{p}\big)$ and  $\big(\frac{j(j-2)}{p}\big) $ respectively.

Moreover there are several similar results in \cite{S20} and our Lemma 2.2 is actually amounts to (1.4) there.  \\ \qed
\end{remark}

To formulate our main theorem it's convenient to make some convention.\\

 Let g(p) be a function whose domain is a subset of the set of all primes. We call g(p) periodic if there is a $m \in \zp$ such that for $p_1 \eq p_2 \pmod{m}$ in the domain we have $g(p_1)=g(p_2).$ In that case we can list its values under $\varphi(m)$ residue classes and g(p) can be evaluated explicitly.\\

In this paper we call two functions g(n) and h(n) of integers equivalent if there is a $m \in \zp $ and for each i with $1\leq i < m,(i,m)=1$ we have a polynomial $f_i$ of degree one with rational coefficients, such that for any n in their common domain with $ n\eq i \pmod{m}$, we have $g(n) = f_i(h(n)).$ Clearly g and h can be evaluated explicitly by each other. \\

Let p be a prime and u,v are integers with $p\nmid u$, we will write v/u for w, where w is the least positive integer with $ wu\eq v \pmod{p}.$
For fixed v/u and let p varies with $p\nmid b$, then we see v/u and $\big(\frac{v/u}{p}\big)$ are well-defined functions of p.\\

\setcounter{theorem}{2}
\begin{theorem}\noindent{\bf (Main theorem) }
Let $p > 3$ be an odd prime and $\big( \frac{}{p }\big)$ denotes Legendre symbol. Let $f(i,j)=ai^2+bij+cj^2$ with $a,b,c \in \Z.$ Write $\Da=b^2-4ac $ and $\sigma=a+b+c.$\\

{\rm (i)}
Let f(i,j) fixed and p varies with $p\nmid c\sigma$. Set $k= -\Da / 4c\sigma.$ Then $\displaystyle \prod_{0<i<j<p/2 \atop p\nmid f(i,j) } \bigg(\frac{ai^2+bij+cj^2}{p}\bigg) \ \ $ and \ \  $ \displaystyle \sum_{y=1}^{p-1} \bigg(\frac{y(y+1)(y+k)}{p}\bigg) \  \pmod{16}$  are equivalent as functions of p .\\ \\

 {\rm (ii)} When $f(i,j)=(mi+nj)(ui+vj)$ with $m,n,u,v\in \Z$, then $\displaystyle \prod_{0<i<j<p/2 \atop p\nmid f(i,j) } \bigg(\frac{f(i,j)}{p}\bigg) \ $
a periodic function of p and can be evaluated explicitly. \\

\end{theorem}

\begin{remark}
To illustrate case (ii) ,we cite Conjecture 7.1 \cite{S20},
$$\displaystyle \prod_{0<i<j<p/2 \atop p\nmid 2i^2 \pm 5ij+2j^2 } \bigg(\frac{2i^2\pm 5ij+2j^2}{p}\bigg) \ \eq \  \frac{1}{2} \bigg(\frac{\pm1}{p}\bigg) \bigg[ \bigg(\frac{-1}{p}\bigg)+\bigg(\frac{2}{p}\bigg)+\bigg(\frac{6}{p}\bigg)+\bigg(\frac{p}{3}\bigg)] $$
Clearly this can be written alternatively by residue of p modulo 24.

In Theorem 2.4, we give explicit formulae for $\displaystyle \prod_{i,j=1 \atop p\nmid (i+j)(si+ j)}^{(p-1)/2} \bigg(\frac{ (i+j)(si+ j) }{p}\bigg) $ and $\displaystyle \prod_{i,j=1 \atop p\nmid (i+j)(si- j)}^{(p-1)/2} \bigg(\frac{ (i+j)(si- j) }{p}\bigg) $, because they will be used in Section 4. \qed
\\ \\
\end{remark}

\proof
For fixed $2 \leq x \leq p-1 $ , by Theorem 2.1 we can determine by $\big(\frac{x-1}{p}\big)$ that there are odd or even number of terms of the form $\big(\frac{f(i,j)}{p}\big)$ with $j\eq ix \pmod{p}$ in the product. Each of these is equal to $\big(\frac{f(1,x)}{p}\big)$ if exists. And when $x=0,1 ,$ no $(i,ix)$ enter into the product. Without of loss of generality we assume $\big(\frac{x-1}{p}\big)=-1$ implies odd terms, then we only need to count those $x \in \{\ 2, \ldots
p-1\}$, for which $\big(\frac{x-1}{p}\big)=\big(\frac{f(1,x)}{p}\big)=-1. $ It suffices to determine the parity of
$$  \sum_{x=2 \atop p \nmid f(1,x)}^{p-1}   \frac{1-\big(\frac{x-1}{p}\big)}{2} \frac{1-\big(\frac{f(1,x)}{p}\big)}{2}.$$

In case (ii), we have $f(i,j)=(mi+nj)(ui+vj)$. We can assume $p \nmid nv$ since otherwise it's easy. Write $r=-m/n,s=-u/v$. It suffices to determine the residue $\displaystyle  \sum_{x=0 }^{p-1}  \bigg(\frac{(x-1)(x-r)(x-s)}{p}\bigg) \pmod{8}$ since other terms in expansion of numerator have explicit formulae. We claim that this residue as a function of p can always be evaluated explicitly. Let's consider
$$  \sum_{x=0 \atop x\neq 1,r,s }^{p-1}   \frac{1+\big(\frac{x-1}{p}\big)}{2}  \frac{1+\big(\frac{x-r}{p}\big)}{2} \frac{1+\big(\frac{x-s}{p}\big)}{2}.$$
This is a integer by definition, and except  $\big(\frac{(x-1)(x-r)(x-s)}{p}\big)$, the value of other terms in expansion of numerator are known. So we can determine the residue $\displaystyle  \sum_{x=0 }^{p-1}   \bigg(\frac{(x-1)(x-r)(x-s)}{p}\bigg) \pmod{8} $ by Legendre symbols of $2,-1,r-1,s-1,r-s$ to p respectively. (See the remark before) By the law of reciprocity, $\displaystyle  \sum_{x=0 }^{p-1}   \bigg(\frac{(x-1)(x-r)(x-s)}{p}\bigg) \pmod{8}$ and consequently $\displaystyle \prod_{0<i<j<p/2 \atop p\nmid f(i,j) } \big(\frac{f(i,j)}{p}\big) \ $ are periodic of p and can be evaluated explicitly.\\

In case (i), when $p\nmid c\sigma$, let $y=f(1,x)=a+bx+cx^2$. We turn to count y instead of x. If y is really in the range of f, it must be  $\big(\frac{(cx+b)^2}{p}\big)=\big(\frac{cy+\Da}{p}\big)=1.$ And if y has two roots, we count y only if $\big(\frac{x_1-1}{p}\big)$ and $ \big(\frac{x_2-1}{p}\big)$ has opposite signs, which implies $\big(\frac{c(\sigma -y)}{p}\big)=-1 .$ Finally we only count those y which are non-residues. As a result, it's enough to determine the parity of
$$ \sum_{y=1 \atop y\neq \sigma,-\frac{\Da}{c} }^{p-1} \frac{1+\big(\frac{y}{p}\big)}{2} \frac{1-\big(\frac{cy+\Da}{p}\big)}{2} \frac{1-\big(\frac{c(\sigma -y)}{p}\big)}{2} .$$
In the expansion of numerator , all other terms except the product of three Legendre symbols have explicit formulae. Simplify the result ,then it is sufficient to determine $ \displaystyle \sum_{y=1}^{p-1} \bigg(\frac{y(y+1)(y+k)}{p}\bigg) \  \pmod{16}$ with $k= -\Da / 4c\sigma.$ The conclusion follows.\\

 \qed

\begin{theorem}
Let $p > 3$ be an odd prime and $\big( \frac{}{p }\big)$ denotes Legendre symbol. Let $f(i,j)=ai^2+bij+cj^2, a,b,c \in \Z .$ \\

{\rm (i)}Let f(i,j) fixed and p varies with $p\nmid ac$. Set $k'=b^2/4ac$. Then \\ $\displaystyle \prod_{i,j=1 \atop p\nmid ai^2+bij+cj^2 }^{(p-1)/2} \bigg(\frac{ ai^2+bij+cj^2 }{p}\bigg) \ $
and \ $ \displaystyle \sum_{y=1}^{p-1} \bigg(\frac{y(y+1)(y+k')}{p}\bigg) \  \pmod{16}$ are equivalent as functions of p. \\

{\rm (ii)} When $f(i,j)=(mi+nj)(ui+vj)$ with $m,n,u,v\in \Z$, then $\displaystyle \prod_{i,j=1 \atop p\nmid f(i,j) }^{(p-1)/2} \bigg(\frac{ f(i,j) }{p}\bigg) \ $ is a periodic function of p and can be evaluated explicitly.\\

Moreover, for $s\in \Z, s \not \eq   0,\pm1 \pmod{p}$, we have
$$\displaystyle \prod_{i,j=1 \atop p\nmid (i+j)(si+j)}^{(p-1)/2} \bigg(\frac{ (i+j)(si+j) }{p}\bigg) \ =
 {\bigg(\frac{s}{p}\bigg)}^{ [ 3+(\frac{s-1}{p}) ] /2 } $$
$$\displaystyle \prod_{i,j=1 \atop p\nmid (i+j)(si-j)}^{(p-1)/2} \bigg(\frac{ (i+j)(si-j) }{p}\bigg) \ =
\bigg(\frac{-1}{p}\bigg)  {\bigg(\frac{-s}{p}\bigg)}^{ [ 1+(\frac{-s-1}{p}) ] /2 } $$

{\rm (iii)} When $f(i,j)= ai^2+cj^2$, then $\displaystyle \prod_{i,j=1 \atop p\nmid ai^2+cj^2 }^{(p-1)/2} \bigg(\frac{ ai^2+cj^2 }{p}\bigg) \ $ is a periodic function of p and can be evaluated explicitly. \\

\end{theorem}

\proof
(iii) follows immediately from (i). To prove (i) we use Gauss' lemma instead of Theorem 2.1 and argue in the same way as Theorem 2.3. \\

When $f(i,j)= (i+j)(si-j)$, by expanding
$ \displaystyle \sum_{x=1 \atop x\neq -1,s }^{p-1}   \frac{1-\big(\frac{x}{p}\big)}{2}  \frac{1-\big(\frac{1+x}{p}\big)}{2} \frac{1+\big(\frac{s-x}{p}\big)}{2},$
we have  $\displaystyle  \sum_{x=1 \atop x\neq -1,s }^{p-1}  \bigg(\frac{(x-1)(x-r)(x-s)}{p}\bigg) $
$$
\eq \begin{cases}  -\big( p+1-(1-(\frac{s}{p}))((1-(\frac{s+1}{p}) ) \big) \pmod{8}  &\t{}\ p\eq1\pmod{4};
\\  -\big( p-7+(1+(\frac{s}{p}))(1-(\frac{s+1}{p}) ) \big) \pmod{8}  &\t{}\ p\eq 3 \pmod{4}.
\end{cases}\ $$

Then we can determine the parity of
$$ \displaystyle \sum_{x=1 \atop x\neq -1,s}^{p-1}   \frac{1\pm\big(\frac{x}{p}\big)}{2} \frac{1-\big(\frac{(1+x)(s-x)}{p}\big)}{2}.$$

And finally we have $\displaystyle \prod_{i,j=1 \atop p\nmid (i+j)(si-j)}^{(p-1)/2} \bigg(\frac{ (i+j)(si-j) }{p}\bigg) $
$$ \eq  \begin{cases}  [ 1-(\frac{s}{p}) ] [  1+(\frac{s+1}{p}) ] /4   \pmod{2}  &\t{}\ p\eq 1\pmod{4};
\\  1+[1+( \frac{s}{p}) ][ 1-( \frac{s+1}{p} ] /4  \pmod{2}  &\t{}\ p\eq 3 \pmod{4}.
\end{cases}\ $$

The case si+j can be proved in the same way.
\qed

\section{Products for quadratic forms}
In this section we always assume f(i,j) to be a quadratic forms with rational coefficients.

Let p be an odd prime and assume that $a_1,\ldots,a_r$ are pairwise incongruent integers modulo p. We will use the notation
$$ F_p(a_1,\ldots,a_r):=\sum_{y=1}^{p} \bigg(\frac{(y+a_1)\ldots(y+a_r)}{p}\bigg) .$$
And we collect some simple properties from exercises in \cite{BEW}, where they use notation $F_r(a_1,\ldots,a_r).$

\setcounter{lemma}{0}
\begin{lemma}\cite[p.\,208]{BEW}
{\rm (i)} For $p\nmid m$, we have $$F_p(0,1,m)= \bigg(\frac{m}{p}\bigg)F_p(0,1,\frac{1}{m})=\bigg(\frac{-1}{p}\bigg)F_p(0,1,1-m). $$
{\rm (ii)}For $p\nmid m$, we have$$F_p(0,1,m^2)= \bigg(\frac{m}{p}\bigg)F_p(0,1,\frac{(m+1)^2}{4m}). $$
{\rm (iii)} For $p\nmid mn$, we have $$  \sum_{y=0}^{p-1}\bigg(\frac{y^2+n}{p}\bigg)\bigg(\frac{y^2+nm}{2}\bigg)=-1+\bigg(\frac{-1}{p}\bigg)F_p(0,1,m). $$\\
\end{lemma}

\medskip

Let $k\in \zp$ and $p\nmid k$. The above lemma tell us if $F_p(0,1,k)\pmod{16}$ is equivalent to $\displaystyle \prod_{0<i<j<p/2 \atop p\nmid f(i,j) } \bigg(\frac{ai^2+bij+cj^2}{p}\bigg) \ $, so are $F_p(0,1,1/k)\pmod{16}$ and $F_p(0,1,1-k)\pmod{16}.$

Conversely, for any $k\in \zp$ there are various $f(i,j)$ for which the products are equivalent to $F_p(0,1,k)  \pmod{16}$.
If for any $f_0(i,j)$ of them we can show that the product $\displaystyle \prod_{0<i<j<p/2 \atop p\nmid f_0(i,j) } \bigg(\frac{f_0(i,j)}{p}\bigg)$ is periodic, so are $F_p(0,1,k)\pmod{16}$ and all products equivalent to it.\\

\setcounter{corollary}{1}
\begin{corollary}
\noindent{\bf $F_p(0,1,2) \ $  }

Let $p>3$ be a prime. Then $\displaystyle \prod_{0<i<j<p/2 \atop p\nmid i^2+j^2 } \bigg(\frac{i^2+j^2}{p}\bigg) \ $ is equivalent to $F_p(0,1,2)\pmod{16} $. Both of them are periodic functions of p and can be evaluated explicitly. \\
\end{corollary}

\proof
Obviously $F_p(0,1,2)$ is the classical Jacobsthal sums and its residue $\pmod{16}$ can be evaluated explicitly. (cf. \cite[p.\,195]{BEW} )\  Moreover we have   (cf. \cite[p.\,3]{S19} )  $$\displaystyle \prod_{0<i<j<p/2 \atop p\nmid i^2+j^2 } i^2+j^2\ \ =
\begin{cases}  (-1)^{\lfloor(p-5)/8\rfloor}\pmod{p}  &\t{when}\ p\eq1\pmod{4};
\\ (-1)^{ \lfloor (p+1)/8 \rfloor } \pmod{p}  &\t{when}\ p\eq 3 \pmod{4}.
\end{cases}$$ .

\qed

\setcounter{theorem}{2}
\begin{theorem}
 {\rm (i)} Let $p > 3$ be a prime, $t=v/u \in \Q$ and $p\nmid u$. We have
$$ \prod_{0<i<j<p/2 \atop p\nmid i^2-ij+tj^2} \bigg(\frac{i^2-ij+tj^2}{p}\bigg) \ =
\begin{cases} -1  &\t{when}\ p\eq5,7\pmod{8} \ and \ \big(\frac{1-4t}{p}\big)=-1;
\\ 1   &\t{otherwise} .
\end{cases}$$. \\

{\rm (ii)}For fixed nonzero $s\in \Z$, let p varies with $p\nmid s.$ Both $F_p(0,1,s^2)\pmod{16} $ and $ \displaystyle \prod_{0<i<j<p/2 \atop p\nmid s^2i^2-j^2 } \bigg( \frac{s^2i^2-j^2 }{p} \bigg)$ are periodic functions of p and can be evaluated explicitly. \\
\end{theorem}

\proof
For fixed j, we have $i^2-ij+tj^2=(i-j/2)^2+\frac{4t-1}{4}$. So $f(i,j)$ are symmetric about the line $j=2i$ in the area $ 0<i<j<p/2 .$ Consequently it's enough to consider the product of $\big(\frac{f(1,2)}{p}\big),\big(\frac{f(2,4)}{p} \big)$ and so on. There are $\lfloor \frac{p-1}{4} \rfloor $ terms, each equals to $\big(\frac{f(1,2)}{p}\big)=\big(\frac{1-4t}{p}\big),$ where $\lfloor x\rfloor$ denotes the largest integer not more than x. By some easy computation we obtain the formula in (i). Moreover , this product is equivalent to $F_p(0,1,(4t-1)^2) /pmod{16}$. For fixed s , set $t=(s+1)/4$. Then $F_p(0,1,s^2)\pmod{16} $ is periodic. And by Theorem 2.3 and Lemma 3.1 it's easy to show $ \displaystyle \prod_{0<i<j<p/2 \atop p\nmid s^2i^2-j^2 } \bigg( \frac{s^2i^2-j^2 }{p} \bigg)$ is equivalent to $F_p(0,1,s^2)\pmod{16} $. The proof is complete.\qed
\\

\begin{theorem}
\noindent{\bf $F_p(0,1,4)\ and \ F_p(0,1,9)$  }
Let prime $p > 3$ varies. \\

{\rm (i)} For $ n=4,-3,9,-8,\frac{1}{4},\frac{3}{4},\frac{4}{3},- \frac{1}{3}$,  \ functions $F_p(0,1,n)\pmod{16}$ are equivalent.\\

{\rm (ii)}As functions of p, $F_p(0,1,4)\pmod{16}$ is equivalent to $\displaystyle \prod_{0<i<j<p/2 \atop p\nmid f(i,j) } \bigg(\frac{f(i,j)}{p}\bigg) \ $ for f(i,j) in the list: \\

$ \textcircled{1} i^2 \pm ij  + j^2  \textcircled{2}2i^2 \pm 5ij + 2j^2 \textcircled{3} 4i^2-j^2 , \  9i^2-j^2  \textcircled{4} 3i^2+j^2 , \  8i^2+j^2 . $  \\

{\rm (iii)}Moreover $F_p(0,1,4)\pmod{16}$ and all equivalent products for various f(i,j) are periodic functions of p and can be evaluated explicitly. \\

\end{theorem}

\proof
By lemma 3.1, we have
$$F_p(0,1,9)=\big(\frac{3}{p}\big)F_p(0,1,\frac{4}{3})=F_p(0,1,\frac{3}{4})=\big(\frac{-1}{p}\big)
F_p(0,1,\frac{1}{4})=\big(\frac{-1}{p}\big)F_p(0,1,4).$$
Hence $F_p(0,1,9)\pmod{16}$ is equivalent to $F_p(0,1,4)\pmod{16}$. All the others are obtained by easy computation. We have shown that $ \prod_{0<i<j<p/2 \atop p\nmid i^2-ij+tj^2} \bigg(\frac{i^2-ij+tj^2}{p}\bigg) $ is periodic, so the same holds for $F_p(0,1,4)\pmod{16}$ and others.\\
\qed

\setcounter{lemma}{4}
\begin{lemma}\cite[Theorem 1.5]{S20}.Let p be an odd prime, we have
$$\displaystyle  \prod_{0<i<j<p/2 \atop p\nmid i^2-ij-j^2} \frac{i^2-ij-j^2}{p} \ \eq
\begin{cases} -5^{(p-1)/4}\pmod{p}  &\t{when}\ p\eq1,9\pmod{20}  ,
\\ (-5)^{(p-1)/4}\pmod{p}   &\t{when}\ p\eq13,17\pmod{20} ,
\\ -1^{ \lfloor \frac{p-10}{20} \rfloor  }\pmod{p}   &\t{when}\ p\eq 3,7\pmod{20} ,
\\ -1^{\lfloor \frac{p-5}{10} \rfloor } \pmod{p}   &\t{when}\ p\eq 11,19\pmod{20} .
\end{cases} $$. \\

\end{lemma}

\setcounter{remark}{4}
\begin{remark}
This result is not easy and utilizes knowledge of values of Lucas sequences modulo primes. \\
Hence we have $$\displaystyle \prod_{i,j=1 \atop p\nmid i^2- ij-j^2 }^{(p-1)/2} \bigg(\frac{ i^2- ij-j^2 }{p}\bigg) \ =
\begin{cases} -1   &\t{when}\ p\eq 13,31,37,39 \pmod{40}
\\ 1  &\t{otherwise} .
\end{cases}
\\ \\  \\  \\ $$
\end{remark}

\setcounter{theorem}{5}
\begin{theorem}
\noindent{\bf $F_p(0,1,5)$ }
Let prime $p > 3$ varies. \\

{\rm (i)}For $n=5,-4,\frac{1}{5},\frac{4}{5},\frac{5}{4},- \frac{1}{4}$ , \ functions $F_p(0,1,n)\pmod{16}$ are equivalent. \\

{\rm (ii)}As functions of p, $F_p(0,1,5)\pmod{16}$ is equivalent to $\displaystyle \prod_{0<i<j<p/2 \atop p\nmid f(i,j) } \bigg(\frac{f(i,j)}{p}\bigg) \ $ for f(i,j) in the list:\\

$
\textcircled{1}  i^2 + ij  - j^2 \
\textcircled{2} i^2 \pm 3ij + j^2  \
\textcircled{3} 4i^2+j^2 , \  i^2+4j^2 \
\textcircled{4} 5i^2-j^2 , \  i^2-5j^2 .$  \\

{\rm (iii)}And $F_p(0,1,5)\pmod{16}$ is also equivalent to
$\displaystyle \prod_{i,j=1 \atop p\nmid i^2- ij-j^2 }^{(p-1)/2} \bigg(\frac{ i^2- ij-j^2 }{p}\bigg). \\$

{\rm (iv)}Moreover $F_p(0,1,5)\pmod{16}$ and all equivalent products for various f(i,j) are periodic functions of p and can be evaluated explicitly. \\

\end{theorem}

\proof
In (iii) we use Theorem 2.4 and we have give the explicit formula of that product. All other results are obtained by easy computation.
\qed \\ \\

At last we give a example that $\displaystyle \prod_{0<i<j<p/2 \atop p\nmid f(i,j) } \bigg(\frac{f(i,j)}{p}\bigg) \ $ is not periodic. By the Theorem 2.3, product for $f(i,j)=i^2+ 4ij +j^2$ is equivalent to $F_p(0,1,3)\pmod{16}$. We can prove that when $p\eq 17 \pmod{24}$, the product equals 1 if and only if 2 is a biquadratic residue modulo p. This is part of Conjecture 7.4 \cite{S20}. We shall investigate periodicity of $F_p(0,1,k)\pmod{16}$ for other k's in the future.\\ \\

\section{Products for linear forms}
Now we turn to linear forms. Let $s \in \Z$ and p be an odd prime. We focus on $ \displaystyle \prod_{i,j=1  \atop p \nmid si +j }^{ (p-1)/2}  \bigg(\frac{si + j}{p} \bigg)$ and $ \displaystyle \prod_{i,j=1  \atop p \nmid si -j }^{ (p-1)/2}  \bigg(\frac{si - j}{p} \bigg)$. Cases $s=1,2$ will be given in 4.2 and 4.7. Sun has given various identities for $ s = 3,4,5,6,8$ respectively as conjectures.\cite{S20} And we shall prove all of them.\\

We shall give identities for general s in Theorem 4.2. The idea is to compare $\displaystyle  \prod_{i,j=1  \atop p \nmid si \pm j }^{ (p-1)/2}  \bigg(\frac{si \pm j}{p} \bigg) \prod_{i,j=1 }^{ (p-1)/2}  \bigg(\frac{i + j}{p} \bigg)  $ with
$\displaystyle \prod_{i,j=1 \atop p\nmid (i+j)(si \pm j)}^{(p-1)/2} \bigg(\frac{ (i+j)(si \pm j) }{p}\bigg) . \\ \\$

Then it's easy to obtain identities for $ \displaystyle \prod_{i,j=1  \atop p \nmid ai +bj }^{ (p-1)/2}  \bigg(\frac{ai + bj}{p} \bigg)$, where $a,b\in\Z. $ And similar results can also be obtained in the same way for $\displaystyle \prod_{0<i<j<p/2 \atop p\nmid ai+bj } \bigg(\frac{ai + bj}{p}\bigg) \ .$  However we shall not pursue them since they will not be used in this paper.\\

If $ p\nmid s$, let $E_p(s)=\{i\in\Z | 0<i<\frac{p}{2}, \{ is\}_p > \frac{p}{2} \}.$ For example, $E_p(4)$ includes those i in $(p/8,p/4) \cup (3p/8,p/2).$  By Gauss' lemma, we know $(-1)^{|E_p(s)|}= \big(\frac{s}{p}\big) $. In general, we can write $E_p(s)$ shortly as $\displaystyle \mathop{\cup}_{k=1}^{ \lfloor s/2  \rfloor } \big(\frac{(2k-1)p}{2s}, \frac{(2k)p}{2s}\big)$, where $\lfloor \ \rfloor$ is the floor function. And clearly $\{1,2, \ldots \frac{p-1}{2}\}$ is the disjoint union of
$E_p(s)$ and $E_p(-s)$. \\

Let $ \# N_p(s) $ denotes the number of non-residues in $E_p(s).$ In other words , $\displaystyle {(-1)}^{\# N_p(s)} \eq  \prod_{ x \in E_p(s) } \bigg(\frac{x}{p}\bigg) $. The relation of $\# N_p(s) $ with class number has been studied before. (For example,\cite[Lemma 12 ]{WC} ). \\

To prove Sun's conjectures we give many alternative identities for ${(-1)}^{\#N_p(s)}$. Most of them are deduced form symmetries of Legendre symbols investigated in \cite{JM}, where they give useful tables about for which class of primes the sum of Legendre symbols over certain intervals is zero. Then it's easy to check that our identities are equivalent to those in \cite{S20}. Thus we have proved all the conjectures there related to linear forms. \\

\setcounter{lemma}{0}

\begin{lemma}
Let $p > 3$ be prime.

$\displaystyle{\rm (i)} \prod_{i,j=1}^{(p-1)/2} \bigg(\frac{i + j}{p}\bigg)   =
\begin{cases}  \big(\frac{2}{p}\big)  &\t{}\ p\eq1 \pmod{4}  ,
\\ \big(\frac{2}{p}\big)(-1)^{(h(-p)+1)/2}  &\t{}\ p\eq3 \pmod{4}.
\end{cases} $

$\displaystyle{\rm (ii)} \prod_{i,j=1 \atop i\neq j }^{(p-1)/2} \bigg(\frac{i - j}{p}\bigg) =
 \begin{cases}  1  &\t{}\ p\eq5 \pmod{8}  ,
\\ -1  &\t{}\ otherwise.
\end{cases} $

\end{lemma}

\proof
{\rm (i)}By symmetry, it is enough to count the product alone the diagonal lines, that is $\displaystyle  \prod_{i=1}^{(p-1)/2} \bigg(\frac{2i}{p}\bigg)$ . When $p\eq1 \pmod{4}$, there are $(p-1)/4$ non-residues in $(0,p/2)$. When $p\eq 3 \pmod{4},$ Mordell \cite{M} noticed $ | \{ 0<k<p/2 :\ \big(\frac{k}{p}\big)=\-1 | \} \eq (-1)^{(h(-p)+1)/2} \pmod{2},$ where h is the class number of $\Q (\root \of{-p})$.   {\rm (ii)} By symmetry again,we need only to count the product of $\displaystyle  \prod_{1\leq i<j \leq (p-1)/2} \bigg(\frac{-1}{p}\bigg) $.   \qed
\setcounter{remark}{0}

\setcounter{theorem}{1}
\begin{theorem}
Let $p>3$ be prime. $s \not \eq 0, \pm1  \pmod{p}$ is a integer.
And let $\#N_p(s)$ be the number of non-residues in $\displaystyle \mathop{\cup}_{k=1}^{ \lfloor s/2  \rfloor } \big(\frac{(2k-1)p}{2s}, \frac{(2k)p}{2s}\big).$
Then
$$  \prod_{i,j=1  \atop p \nmid si + j }^{ (p-1)/2}  \bigg(\frac{si + j}{p} \bigg) =
\begin{cases} \big(\frac{2}{p}\big) (-1)^{\#N_p(s)}   &\t{}\ p\eq1 \pmod{4}  ,
\\  \big(\frac{2s}{p}\big)(-1)^{ \#N_p(s)+(h(-p)+1)/2 }  &\t{}\ p\eq3 \pmod{4}.
\end{cases} $$

$$  \prod_{i,j=1  \atop p \nmid si - j }^{ (p-1)/2}  \bigg(\frac{si - j}{p} \bigg) =
\begin{cases} \big(\frac{s}{p}\big) (-1)^{\#N_p(s)}   &\t{}\ p\eq1 \pmod{4}  ,
\\ - \big(\frac{2}{p}\big)(-1)^{ \#N_p(s) }  &\t{}\ p\eq3 \pmod{4}.
\end{cases} $$

\end{theorem}

\proof
It suffices to prove the case $(si-j)$. The idea is to multiply $(si-j)$ by $(i+j)$ to obtain an quadratic forms.
$\displaystyle  \prod_{i,j=1  \atop p \nmid si - j }^{ (p-1)/2}  \bigg(\frac{si - j}{p} \bigg) \prod_{i,j=1 }^{ (p-1)/2}  \bigg(\frac{i + j}{p} \bigg)  $ differs from
$\displaystyle \prod_{i,j=1 \atop p\nmid (i+j)(si-j)}^{(p-1)/2} \bigg(\frac{ (i+j)(si-j) }{p}\bigg) \ $
in that those $\big( \frac{i+j}{p} \big)$ with $si \eq j\pmod{p}$ are counted in the former but not the latter. So we have

$$\displaystyle  \prod_{i,j=1  \atop p \nmid si - j }^{ (p-1)/2}  \bigg( \frac{si - j}{p} \bigg) \prod_{i,j=1 }^{ (p-1)/2}  \bigg(\frac{i + j}{p} \bigg) = \displaystyle \prod_{i,j=1 \atop p\nmid (i+j)(si-j)}^{(p-1)/2} \bigg(\frac{ (i+j)(si-j) }{p}\bigg) \ \displaystyle \prod_{i,j=1  \atop si \eq j \pmod{p} }^{ (p-1)/2}  \bigg( \frac{i+j}{p} \bigg) $$

The first product on the right is given by Theorem 2.4 , and
 $$\displaystyle \prod_{i,j=1  \atop si \eq j \pmod{p} }^{ (p-1)/2}  \bigg(\frac{i+j}{p} \bigg) = \prod_{i \in E_p(-s)}  \bigg(\frac{(s+1)i}{p} \bigg) = \bigg(\frac{s+1}{p} \bigg)^{|E_p(-s)|}\prod_{i \in E_p(-s)}  \bigg(\frac{i}{p} \bigg)  .
 $$
We know that $(-1)^{|E_p(-s)|}= \big(\frac{s}{p}\big) $, as given by Gauss' lemma. Substitute and simplify, the proof is complete                                                              \qed

$\\$

As a result, we have
$$  \prod_{i,j=1  \atop p \nmid 4i - j }^{ (p-1)/2}  \bigg(\frac{4i - j}{p} \bigg) =
\begin{cases}  (-1)^{\#N_p(4)}   &\t{}\ p\eq1 \pmod{4}  ,
\\  \big(\frac{-2}{p}\big)(-1)^{ \#N_p(4) }  &\t{}\ p\eq3 \pmod{4}.
\end{cases} $$
However conjecture 7.7 in \cite{S20} takes another form.
$$  \prod_{i,j=1  \atop p \nmid 4i - j }^{ (p-1)/2}  \bigg(\frac{4i - j}{p} \bigg) =
\begin{cases}  (-1)^{\frac{p-1}{4}}   &\t{}\ p\eq1 \pmod{4}  ,
\\  (-1)^{ \lfloor p/8 \rfloor  }  &\t{}\ p\eq3 \pmod{4}.
\end{cases} $$

Next theorem tell us they are the same. And we are going to give all that are needed to show that our identities for $si\pm j ,s=3,4,5,6,8$ actually equivalent to those in \cite{S20}.

Let's make some convention. In the following p is always an prime and $ p>10$ . For any finite set $  A  = \{a_1,a_2,\ldots ,a_n \} \subset Z$ ,we will denote $\displaystyle  \bigg(\frac{\prod }{p} \bigg)A :=  \prod_{i=1}^{n} \bigg(\frac{a_i}{p} \bigg).$  We will also write $(p/n,p/m)$ for $ \{ i \in \Z |  p/n< i < p/m    \}. $ For example, by Mordell's result \cite{M} mentioned above when $p\eq 3\pmod{4}$, then $\big(\frac{\prod }{p} \big)(0,p/2) = {(-1)}^{\frac{h(-p)+1}{2}}.$

By definition ${(-1)}^{\#N_p(4)}=\big(\frac{\prod }{p} \big)(p/8,p/4) \cup (3p/8,p/2)$.\\
\begin{theorem}
$$  {(-1)}^{\#N_p(4)} =
\begin{cases} 1   &\t{when}\ p=1+8k  ,
\\  -1  &\t{when}\ p=5+8k.
\\ {(-1)}^{k}   &\t{when}\ p=3+8k.
\\  {(-1)}^{k+1}   &\t{when}\ p=7+8k.
\end{cases}
$$
\end{theorem}

\begin{proof}
 Consider $\big(\frac{\prod }{p} \big) (p/8,3p/8).$ Take the case $p=3+8k$ as example, then $(p/8,3p/8)=\{k+1,k+2,\ldots,3k+1\}.$ Since $\bigg(\frac{2}{p} \bigg)=-1$ ,so we have
$$\big(\frac{\prod }{p} \big) (p/8,3p/8) = {(-1)}^{2k+1} \big(\frac{\prod }{p} \big) \{2k+2,2k+4,\ldots,6k+2\}$$ .

Next since $\bigg(\frac{-1}{p} \bigg)=-1$ we have
$$
\bigg(\frac{\prod }{p} \bigg) \{2k+2,2k+4,\ldots,6k+2\} = \bigg(\frac{\prod }{p} \bigg) \{2k+2,2k+4,\ldots,4k\} \cup \{4k+2,\ldots,6k+2\} $$
$$= {(-1)}^{k+1} \bigg(\frac{\prod }{p} \bigg) \{2k+2,2k+4,\ldots,4k\} \cup \{2k+1,2k+3,\ldots,4k+1\} $$

$$= {(-1)}^{k+1} \bigg(\frac{\prod }{p} \bigg) \{2k+1,2k+2,\ldots,4k\} ={(-1)}^{k+1} \bigg(\frac{\prod }{p} \bigg) (p/4,p/2).$$
So finally we arrive at the relation
$${(-1)}^{\#N_p(4)}= \bigg(\frac{\prod }{p} \bigg)(p/3,3p/8)  \bigg(\frac{\prod }{p} \bigg)(p/4,p/2)
= {(-1)}^{3k+2}={(-1)}^{k}
$$
The other cases are proved in exactly the same way. \\
\end{proof}

The idea of this proof will be repeatedly used in the following theorems. First we choose some union of subintervals symmetric about p/4 , multiply every integer inside by 2, then substitute those $x>p/2$ with $p-x$. \\

Next we turn to ${(-1)}^{\#N_p(2)}=\big(\frac{\prod }{p} \big)(p/4,p/2).$

\setcounter{lemma}{3}
\begin{lemma}\cite[p.\,972(1.3)(1.4)]{BEW}   Let $p\eq 1 \pmod{4}$ be a prime. Then
$$ \bigg(\frac{\prod }{p} \bigg)(0,p/4)  =
\begin{cases}
  {(-1)}^{(p-1)/8+h(-4p)/4}   &\t{}\  p\eq 1\pmod{8},
\\ {(-1)}^{(p-5)/8+(h(-4p)-2)/4}  &\t{}\ p\eq 5\pmod{8}.
\end{cases}$$

\end{lemma}

\begin{lemma}\cite{JM}
Let $p\eq 3 \pmod{4}$ be a prime. Define $\displaystyle S_r^{n}:= \sum_{a \in (\frac{(r-1)p}{n},\frac{rp}{n})} \bigg(\frac{a }{p}\bigg).$
\\Then \  \  \ $
\begin{cases}   S_1^{4}=0 &\t{}\  p\eq 3\pmod{8},
\\  S_2^{4}=0  &\t{}\ p\eq 7\pmod{8}.
\end{cases}$ \\
\end{lemma}

\setcounter{corollary}{5}
\begin{corollary}

{\rm (i)} When $p\eq 1 \pmod{4}$,
$ {(-1)}^{\#N_p(2)}= \big( \frac{2}{p} \big) \big(\frac{\prod }{p} \big)(0,p/4).$  \\

{\rm (ii)}
 When $p=3+8k,$   $\begin{cases}   \big(\frac{\prod }{p} \big)(0,p/4)=  {(-1)}^{k},
\\  {(-1)}^{\#N_p(2)}=  {(-1)}^{k+\frac{h(-p)+1}{2}},
\\  \big(\frac{\prod }{p} \big)(0,p/8)\cap (3p/8,p/2)= 1. \end{cases} \\ \\ $

{\rm (iii)} When $p=7+8k,$ $\begin{cases} \big(\frac{\prod }{p} \big)(0,p/4)=  {(-1)}^{k+1+\frac{h(-p)+1}{2}},
\\ {(-1)}^{\#N_p(2)} =  {(-1)}^{k+1},
\\ \big(\frac{\prod }{p} \big)(p/8,3p/8) =1 .
\end{cases}$ \\
\end{corollary}

\begin{proof}
(i) When $p\eq 1 \pmod{4}$ , there are $(p-1)/4$ non-residues in (0,p/2). (ii) When $p=3+8k, (0,p/4) = \{1,2,\ldots,2k\}.$ So by the lemma there are $k=(p-3)/8$ non-residues in it. And $\big(\frac{\prod }{p} \big)(0,p/8)\cap (3p/8,p/2)=  {(-1)}^{\#N_p(4)} \big(\frac{\prod }{p} \big)(0,p/4) $,since ${(-1)}^{\#N_p(4)}=\big(\frac{\prod }{p} \big)(p/8,p/4) \cup (3p/8,p/2)$. (iii) are proved similarly. \\
\end{proof}

\begin{corollary} Let p be an odd prime.
$$  \prod_{i,j=1  \atop p \nmid 2i + j }^{ (p-1)/2}  \bigg(\frac{2i + j}{p} \bigg) =  \bigg(\frac{\prod }{p} \bigg)(0,p/4) =
\begin{cases}   {(-1)}^{k+h(-4p)/4}   &\t{when}\  p=1+8k ,
\\ {(-1)}^{k+(h(-4p)-2)/4}  &\t{when}\ p=5+8k,
\\  {(-1)}^{k}    &\t{when}\ p=3+8k ,
\\ {(-1)}^{k+1+\frac{h(-p)+1}{2}} &\t{when}\ p=7+8k.
\end{cases} $$

$$  \prod_{i,j=1  \atop p \nmid 2i - j }^{ (p-1)/2}  \bigg(\frac{2i - j}{p} \bigg) = \bigg(\frac{-2}{p}\bigg) (-1)^{\#N_p(s)} =
\begin{cases}   {(-1)}^{k+h(-4p)/4}   &\t{when}\  p=1+8k ,
\\ {(-1)}^{k+(h(-4p)-2)/4}  &\t{when}\ p=5+8k,
\\  {(-1)}^{k+\frac{h(-p)+1}{2}}   &\t{when}\ p=3+8k ,
\\  {(-1)}^{k} &\t{when}\ p=7+8k.
\end{cases} \\$$  \\
\end{corollary}

\setcounter{remark}{6}
\begin{remark}
In fact, congruences involving class number and $\frac{p-1}{2}!!$ for $ \displaystyle \prod_{i,j=1  \atop p \nmid 2i +\epsilon j }^{ (p-1)/2}  2i +\epsilon j$ (where $\epsilon=\pm1$) has been obtained by Sun and proved by Fedor Petrov , see Question 314331 in mathoverflow.net.  \qed  \\

\end{remark}

Now turn to $s=8$. By definition ${(-1)}^{\#N_p(8)}= \big(\frac{\prod }{p} \big)(p/16,p/8) \cup (3p/16,p/4) \cup (5p/16,3p/8) \cup (7p/16,p/2).$

\setcounter{theorem}{6}
\begin{theorem}

$$  {(-1)}^{\#N_p(8)}=
\begin{cases} 1  &\t{when}\  p\eq 7\pmod{8} ,
\\  \big(\frac{\prod }{p} \big)(p/4,p/2)  &\t{otherwise}\ .
\end{cases}
$$
\end{theorem}

\begin{proof}
Consider $(0,p/16) \cup (p/8,3p/16) \cup (5p/16,3p/8) \cup (7p/16,p/2)$ and proceed as the proof of Theorem 4.3. Then we arrive at the relation between ${(-1)}^{\#N_p(8)}\big(\frac{\prod }{p} \big)(p/4,p/2)$ and  ${(-1)}^{\#N_p(4)} $
\end{proof}

Now it's easy to check that our identities by Theorem 4.2 equivalent to those conjectures for $4i\pm j, 8i\pm j$ in \cite{S20}.
Corollary 4.6 is use in 4i+j for $p \eq 7 \pmod{8}$ and in proving theorem 4.7 for $8i\pm j$.

Then we turn to $s=5.$

\setcounter{lemma}{7}
\begin{lemma}\cite{JM}
Define $\displaystyle S_r^{n}:= \sum_{a \in (\frac{(r-1)p}{n},\frac{rp}{n})} \bigg(\frac{a }{p}\bigg).$
Then we have  \\ {\rm (i)} $S_1^{10}=0$ for $p\eq 3,27 \pmod{40}.$ \\
                 {\rm (ii)} $S_1^{10}+ S_3^{10}+S_5^{10}=0 $ for $p\eq 7,23 \pmod{40}.$
\end{lemma}
\setcounter{theorem}{8}
\begin{theorem}
By definition ${(-1)}^{\#N_p(5)}= \big(\frac{\prod }{p} \big)(p/10,p/5) \cup (3p/10,4p/5) . $
{\rm (i)} When $p\eq 1\pmod{4}$
$$ {(-1)}^{\#N_p(5)} =
\begin{cases}  \big(\frac{\prod }{p} \big)(0,p/10)  &\t{}\  p\eq 9,13\pmod{20},
\\ \big(\frac{2}{p}\big) \big(\frac{\prod }{p} \big)(0,p/10) &\t{}\ p\eq 1,17\pmod{20}.
\end{cases}$$

{\rm (ii)} When $p\eq 3\pmod{4}$
$$ {(-1)}^{\#N_p(5)} =
\begin{cases} \big(\frac{5}{p}\big)\big(\frac{\prod }{p} \big)(p/10,p/2)   &\t{}\  p\eq 3,19\pmod{20},
\\  \big(\frac{10}{p}\big)\big(\frac{\prod }{p} \big)(p/10,p/2) &\t{}\ p\eq 7,11\pmod{20}.
\end{cases}$$

{\rm (iii)}
$$  \bigg(\frac{\prod }{p} \bigg)(0,p/10)=
\begin{cases}  1  &\t{}\  p\eq 3\pmod{20},
\\  \big(\frac{2}{p}\big) &\t{}\ p\eq 7\pmod{20}.
\end{cases}$$

\end{theorem}

\begin{proof}
Consider $(p/5,3p/10)$ and proceed as the proof of Theorem 4.3. (iii) is the corollary of the lemma.
\end{proof}

Finally we deal with s=3,6.

\setcounter{lemma}{9}
\begin{lemma}\cite{JM}
Define $\displaystyle S_r^{n}:= \sum_{a \in (\frac{(r-1)p}{n},\frac{rp}{n})} \bigg(\frac{a }{p}\bigg).$
Then we have \\
                   {\rm (i)} $S_2^{6}=0$ for  $  p\eq 11 \pmod{12}. $\\
                  {\rm (ii)} $ S_1^{6}+ S_3^{6}=0 $ for  $ p\eq 7 \pmod{12}.$ \\
\end{lemma}

\setcounter{theorem}{9}
\begin{theorem}
$${(-1)}^{\#N_p(3)}= \bigg(\frac{\prod }{p} \bigg)(p/6,p/3)=
\begin{cases}  \big(\frac{2}{p}\big)\big(\frac{\prod }{p} \big)(0,p/3)  &\t{}\  p\eq 1\pmod{12},
\\  \big(\frac{\prod }{p} \big)(0,p/3)   &\t{}\ p\eq 5\pmod{12}.
\\  -\big(\frac{\prod }{p} \big)(p/3,p/2)  &\t{} p\eq 7\pmod{12},
\\ \big(\frac{2}{p}\big)\big(\frac{\prod }{p} \big)(p/3,p/2)  &\t{} p\eq 11\pmod{12},
\end{cases}$$

\end{theorem}

\setcounter{lemma}{10}
\begin{lemma}\cite{JM}

                   {\rm (i)} $S_2^{12}=S_4^{12}=S_6^{12} , \ \  S_3^{12}=S_5^{12} $ for  $  p\eq 1 \pmod{24}. $\\
                  {\rm (ii)} $ S_2^{12}+ S_4^{12}+ S_6^{12}=0 $ for  $ p\eq 19 \pmod{24}.$  \\
                   {\rm (iii)}$ S_2^{12}+S_6^{12}=S_4^{12}=0 $ for  $ p\eq 23 \pmod{24}.$ \\
                   {\rm (iv)} $ S_1^{12}+ S_3^{12}+ S_5^{12}=0 $ for  $ p\eq 7 \pmod{24}.$ \\
                   {\rm (v)} $ S_2^{12}=S_5^{12} , \  \ S_2^{4}=S_4^{12}+ S_5^{12}+ S_6^{12}=0  $ for  $ p\eq 11 \pmod{24}.$ \\
\end{lemma}
\setcounter{theorem}{11}
\begin{theorem}
By definition ${(-1)}^{\#N_p(6)}= \big(\frac{\prod }{p} \big)(p/12,p/6) \cup (p/4,p/3) \cup (5p/12,p/2)$. Then we have

{\rm (i)}
$$ {(-1)}^{\#N_p(6)} =
\begin{cases}  \big(\frac{\prod }{p} \big)(p/4,3/p) &\t{}\  p\eq 1\pmod{4},
\\  \big(\frac{3}{p}\big)\big(\frac{\prod }{p} \big)(p/4,3/p) &\t{}\ p\eq 3\pmod{4}.
\end{cases}$$

{\rm (ii)}
$$ {(-1)}^{\#N_p(6)} =
\begin{cases}  \big(\frac{\prod }{p} \big)(0,12/p)  &\t{when}\  p=1+24k,
\\  {(-1)}^{k+1}  &\t{when}\ p=19+24k,
\\  {(-1)}^{k+1}  &\t{when}\ p=23+24k,
\\ {(-1)}^{k}\big(\frac{\prod }{p} \big)(0,2/p)  &\t{when}\  p=7+24k,
\\  {(-1)}^{k+1}\big(\frac{\prod }{p} \big)(0,2/p)  &\t{when}\ p=11+24k.
\end{cases}$$
\end{theorem}

\begin{proof}
Consider $(0,p/12)\cap (5p/12,p/2)$ and proceed as in the proof of Theorem 4.3. (ii) are all corollary of lemma 4.11. When p=1+24k, every interval $((t-1)p/12,tp/12), 1\leq t \leq 12 $ have same number of integers. By lemma we have $S_3^{12}=S_5^{12},$ so these two intervals have same number of nonresidues.Therefore $\big(\frac{\prod }{p} \big)(p/6,p/4)= \big(\frac{\prod }{p} \big)(p/3,5p/12).$ Other cases are proved similarly as before.

\end{proof}

Given these results it's trivial to check that we actually proved conjecture 7.6 - 7.10\cite{S20}. And since we've actually done so case by case, we will stop here.

\subsection*{Acknowledgements}
The author is deeply grateful to Zhi-Wei Sun for his guidance.

\end{document}